\documentclass[12pt]{amsart}
\usepackage[utf8]{inputenc}
\usepackage{amsthm}
\usepackage{amsmath}
\usepackage{tikz-cd}
\usepackage{tikz}
\usepackage{enumerate}
\usepackage{ amssymb }

\title{Linearly Mismatched Free-by-Cyclic Groups are Asynchronously Automatic}
\author[B. L. Jeffers]{Benjamin L. Jeffers}
\address{Trinity University, 1 Trinity Place, San Antonio, TX 78212}
\email{bjeffers@trinity.edu}

\author[B. Gustafson]{Benjamin Gustafson}
\address{Trinity University, 1 Trinity Place, San Antonio, TX 78212}
\email{bgustafs@trinity.edu}

\theoremstyle{plain}
\newtheorem{theorem}{Theorem}
\newtheorem{corollary}[theorem]{Corollary}

\theoremstyle{definition}
\newtheorem{definition}[theorem]{Definition}
\newtheorem*{remark}{Remark}
\newtheorem{example}{Example}
\newtheorem*{ack}{Acknowledgment}

\newcommand{\R}{\mathbb{R}}

\newcommand{\Z}{\mathbb{Z}}

\newcommand{\abs}[1]{\left\vert#1\right\vert}
\newcommand{\set}[1]{\left\{#1\right\}}

\begin{document}

\begin{abstract}
We call the family of free-by-cyclic groups defined by $G = \left< a, t, b_1, b_2, \ldots b_k \mid at = ta, b_1^{-1}tb_1 = a^{n_1}t, \ldots b_k^{-1}tb_k = a^{n_k}t \right>$ for $n_1, n_2, \ldots n_k \in \Z$ linearly mismatched since the automorphisms used to define the HNN extensions grow linearly at different rates. Using techniques from Elder's thesis, \cite{Elder}, namely words with a parallel stable letter structure, we prove that linearly mismatched free-by-cyclic groups are asynchronously automatic, and thus they have a solvable word problem. 
\end{abstract}
\maketitle\thispagestyle{empty}

\section{Introduction}
Automatic groups were introduced by William Thurston in 1986, after he noticed that some results of Jim Cannon on geometric properties of the Cayley graph of cocompact discrete hyperbolic groups could be restated in terms of finite state automata \cite{Holt}. The theory has remained an area of interest because of its strong connection to algorithmic properties of groups. For example, if a group has an automatic structure, then its word problem is solvable in quadratic time and it has a quadratic isoperimetric inequality \cite{WP}.  Another class of groups that has a sustained interest are biautomatic groups, they are defined by a stronger condition than the class of automatic groups. Since automatic groups were introduced it has been a question of whether or not every automatic group has a biautomatic structure, see Open Question 2.5.6 in \cite{WP}. It has been claimed that the group \[
G^{1, 2} = \left< a, b, c, t \mid at = ta, b^{-1}tb = at, c^{-1}tc = a^2t \right>
\]
is not biautomatic and it is a candidate for a group that is automatic but not biautomatic. However, the result was never published. We call such a group \emph{linearly mismatched} because it is formed by two HNN extensions, each of which is growing linearly, but at different rates. Indeed, the functions forming the HNN extensions are $\phi_1$ defined by $\phi_1\left(a\right) = a$, and $\phi_1\left(t\right) = at$ while the function defining the second HNN extension is $\phi_2\left(a\right) = a, \phi_2\left(t\right) = a^2t$, and thus we have \[
\phi_1^n(t) = a^nt, \;\; \phi_2^n(t) = a^{2n}t.
\] This phenomena can also be seen in a Cayley graph of the group, where there are planes corresponding to $a$ and $t$ and there is a strip corresponding to $b$ that comes straight out of the $at$-plane, and a strip corresponding to $c$ that comes up, but they are attached along different lines, as in figure \ref{cayley}. 

We demonstrate a weaker result than automaticity here, proving that the group $G^{1,2}$ is \emph{asynchronously} automatic. More generally, we show that the family of groups \[G
 = \left< a, t, b_1, b_2, \ldots b_k \mid at = ta, b_1^{-1}tb_1 = a^{n_1}t, b_2^{-1}tb_2 = a^{n_2}t, \ldots b_k^{-1}tb_k = a^{n_k}t \right>
\] 
for $n_1, n_2, \ldots n_k \in \Z$
are asynchronously automatic. Asynchronously automatic groups still have strong algorithmic properties. For example, an asynchronously automatic group satisfies an exponential isoperimetric inequality \cite{WP}, and they have a solvable word problem, although it is unknown if it is solvable in quadratic time \cite{Holt}. The class of asynchronously automatic groups is strictly weaker than the class of automatic groups. There are examples of groups that are asynchronously automatic but not automatic such as the Baumslag-Solitar groups \[
G_{p, q} = \set{x, y \mid yx^py^{-1} = x^q}
\] for $p \neq q$. In \cite{WP} it is shown that $G_{p, q}$ with $p \neq q$ is asynchronously automatic but not automatic by proving that their isoperimetric inequality is exponential. Unfortunately the same technique does not work to show that $G^{1, 2}$ is not automatic since it has a quadratic isoperimetric inequality \cite{BB00}. \par Previously, Gersten showed that the group  $G^{1, 2}$ is not CAT(0), \cite{Ger94}. It is also an open question of whether or not every CAT(0) group is biautomatic or automatic. Recently Leary and Minasyan found examples of CAT(0) groups that are not biautomatic, \cite{Leary-Min}. In 1997 Niblo and Reeves made significant progress towards this question by showing that all groups acting effectively, cellularly, properly discontinuously and cocompactly on a simply connected, non-positively curved cube complex have a synchronous biautomatic structure, \cite{NibloReeves}. It is also unknown which free-by-cyclic groups are automatic and biautomatic. 

\par To prove our main result we follow the techniques of Murray Elder in his thesis \cite{Elder} and use words that have a parallel stable letter structure. As a result of our main theorem we obtain that linearly mismatched free-by-cyclic groups have a solvable word problem. Although Schleimer showed in \cite{SS08} that free-by-cyclic groups already have a solvable word problem, we give a different proof using the fact that asynchronously automatic groups have a solvable word problem \cite{WP}.

\vspace{0.5cm}
\begin{ack}
We would like to thank Nata\v sa Macura for the suggestion of the problem and many helpful discussions. 
\end{ack}


\section{Preliminaries}
\subsection{Normal Forms}

A more in depth discussion of this section can be found in \cite{Lyndonn-Schupp}. 

\begin{definition}
Let $G$ be a group and let $I$ be an indexed set. Let $\set{H_i \mid i \in I}$ and $\set{K_i \mid i \in I}$ be families of subgroups of $G$ with $\set{\phi_i \mid i \in I}$ a family of maps such that each $\phi_i \colon H_i \to K_i$ is an isomorphism. Then the \emph{HNN extension with base $G$, stable letters $b_i, i \in I$, and associated subgroups $H_i$ and $K_i, i \in I$}, is the group \[
G^* = \left< G, b_i \left(i \in I\right) \mid b_i^{-1}h_ib_i = \phi_i\left(h_i\right), h_i \in H_i \right>
\]
\end{definition}

An HNN extension can be interpreted topologically as follows. More discussion of this topological interpretation can be found in \cite{Lyndonn-Schupp}. Let $X$ be a pathwise connected topological space and let $U$ and $V$ be subspaces of $X$ and let $h \colon U \to V$ be a homeomorphism. Let $I$ be the unit interval and let $C = U \times I$. Identify $U \times \set{0}$ with $U$ and identify $U \times \set{1}$ with $V$ by $h$. Let $Z$ be the resulting space. Let $\pi_1\left(X\right)$ be the fundamental group of $X$. Then by the Siefert-van Kampen theorem, the fundamental group of $Z$ can be written as \[
\pi_1\left(Z\right) = \left< \pi_1\left(X\right) , t \mid t^{-1}\pi_1\left(U\right)t = \pi_1\left(V\right) \right>.
\] So essentially we are extending a space $U$ by the unit interval and identifying the ends to $U$ and $V$ through the homeomorphism $h$, and the resulting fundamental group is an HNN extension. For this reason, these groups are also sometimes called \emph{mapping tori}, since we form a torus-like object, but instead of identifying the sides of a square, we identify ends of a space through a homeomorphism. 

\begin{example}\label{example1}
Let $G = \langle a, b, c \rangle$, the nonabelian free group on three generators. Then let $\phi\left(a\right) = a, \phi\left(b\right) = ba, \phi\left(c\right) = ca^2$. This is an automorphism of $G$, and thus our associated subgroups are $G$ and $G$. This gives \[
G^* = \langle a, b, c, t \mid tat^{-1} = a, tbt^{-1} = ba, tct^{-1} = ca^2 \rangle. 
\]
\end{example}

From now on the letter $g$, with or without subscripts will denote an element of $G$, and if it is thought of as a word in $G$, then it is a word on the generators of $G$, meaning it contains no occurrences of $t^{\pm 1}$. The letter $\epsilon$, with or without subscripts, will denote $\pm 1$. 

\begin{definition}
A sequence $g_0, b_{i_1}^{\epsilon_1}, g_1, \ldots , b_{i_n}^{\epsilon_n} g_n$ is said to be \emph{reduced} if there is no consecutive subsequence $b_{i_k}^{-1}, g_i, b_{i_k}$ with $g_i \in H_i$, or $b_{i_m}g_j b_{i_m}^{-1}$ with $g_j \in K_j$ for $i, j > 0$.
\end{definition}

\begin{example}
In the group $G = \langle a, b, c, t \mid tat^{-1} = a, tbt^{-1} = ba, tct^{-1} = ca^2 \rangle $, the word $a^2tbt^{-1}c^2t$ is not reduced since it contains $tbt^{-1}$, while the word $a^2bac^2t$ is reduced. 
\end{example}

\begin{definition}
A \emph{normal form} is a sequence $g_0, b_{n_1}^{\epsilon_1}, \ldots , b_{n_k}^{\epsilon_k}, g_k$ where \begin{enumerate}
    \item $g_0$ is an arbitrary element of $G$,
    \item If $\epsilon_i = -1$, then $g_i$ is a representative of a coset of some $H_i$ in $G$,
    \item if $\epsilon_i = +1$, then $g_j$ is a representative of a coset of some $K_j$ in $G$, and 
    \item there is no consecutive subsequence $b_i^\epsilon, 1, b_i^{\epsilon}$
\end{enumerate}
\end{definition}

\begin{example}
In $G = \langle a, b, c, t \mid tat^{-1} = a, tbt^{-1} = ba, tct^{-1} = ca^2 \rangle$, the word $a^2bac^2ta^{-2}b$ is a normal form. 
\end{example}

This leads us to the normal form theorem which we will use to prove that our family of groups is asynchronously automatic. 

\begin{theorem}[The Normal Form Theorem for HNN Extensions] Let $G^* = \langle G, b_1, \ldots, b_n \mid b_i^{-1}h_ib_i = \phi\left(h_i\right), h_i \in H_i \rangle$ be an HNN extension. Then 
\begin{enumerate}
    \item The group $G$ is embedded in $G^*$ by the map $g \to g$. If $g_0b_{i_i}^{\epsilon_1} \cdots b_{i_n}^{\epsilon_n} g_n = 1$ in $G^*$ with $n \geq 1$, then $g_0, b_{i_1}^{\epsilon_1}, \ldots , b_{i_n}^{\epsilon_n}, g_n$ is not reduced. 
    \item Every element $w$ of $G^*$ has a unique representation as $w = g_0b_{i_1}^{\epsilon_1} \cdots b_{i_n}^{\epsilon_n}g_n$ where $g_0, b_{i_1}^{\epsilon_1}, \ldots , b_{i_n}^{\epsilon_n}, g_n$ is a normal form. 
\end{enumerate}
\end{theorem}


\subsection{Automatic Groups}
This section is discussed in much greater detail in \cite{WP}. Throughout, we take $\epsilon$ to be the empty word. For a language $L$ over an alphabet $A$, say the generators of a group $G$, we obtain a map $\pi \colon L\to G$ where a word $w_1w_2 \cdots w_n$ is reduced according to the relations of $G$. We write the image of a word $w$ under $\pi$ as $\overline{w}$. When we speak of a path in the Cayley graph corresponding to a word $w$, we denote it by $\widehat{w}$. Finally we use $w\left(t\right)$ to mean the first $t$ letters of a word $w$.

\begin{definition}
Let $G$ be a group. A \emph{synchronous automatic structure} on $G$ consists of a set $A$ of semigroup generators of $G$, a finite state automaton $W$ over $A$, and finite state automata $M_x$ over $\left(A, A\right)$, for $x \in A \cup \set{\epsilon}$ satisfying the following conditions:
\begin{enumerate}
    \item The map $\pi \colon L\left(W\right) \to G$ is surjective.
    \item For $x \in A \cup \set{\epsilon}$, we have $\left(w_1, w_2\right) \in L\left(M_x\right)$ if and only if $\overline{w_1x} = \overline{w_2}$ and both $w_1$ and $w_2$ are elements of $L\left(W\right)$. 
\end{enumerate} The automaton $W$ is called the \emph{word acceptor}, $M_{\epsilon}$ the \emph{equality recognizer}, and each $M_x$, for $x \in A$, a \emph{multiplier automaton}. The equality recognizer takes in two words and accepts the pair, if and only if the words are equal. Similarly, $M_x$ takes in a pair of words and accepts them if and only if they are equal up to a right multiplication by $x$. A \emph{synchronous automatic group} is one that admits a synchronous automatic structure. 
\end{definition}

There is another equivalent definition that will be very useful for us.
\begin{theorem}
Let $G$ be a group and let $A$ be a finite set of semigroup generators for $G$. Let $W$ be a finite state automaton over $A$ and suppose that $\pi \colon L\left(W\right) \to G$ is surjective. Then $A$ and $W$ are part of a synchronous automatic structure on $G$ if and only if there is a number $k$ with the property that whenever two strings $w_1$ and $w_2$ accepted by $W$ are such that $\overline{w_1x} = \overline{w_2}$ for some $x \in A$, the corresponding paths $\widehat{w_1}$ and $\widehat{w_2}$ are a uniform distance less than $k$ apart. 
\end{theorem}

We will be proving that a specific group has a slightly weaker structure than an automatic one. Namely, we will show that our linearly mismatched free-by-cyclic is \emph{asynchronously automatic}. The difference is that when we look at the paths $\widehat{w}$ and $\widehat{wx}$ for some $x \in A$, we allow them to travel at different speeds so that they stay within a bounded distance of each other.

\begin{definition}
Paths $w_1$ and $w_2$ are said to \emph{$k$-fellow travel} if $d\left(w\left(t\right), u\left(t\right)\right) \leq k$ for each $t \in \R$ with $t \geq 0$. The two paths are \emph{asynchronous} $k$-fellow travelers if there is a non-decreasing proper continuous function $\phi \colon [0, \infty) \to [0, \infty)$ such that $d\left(w\left(t\right)\right), u\left(\phi\left(t\right)\right)) \leq k$. This means that any point on $w$ is within $k$ of some point on $u$ and vise versa. A language $L$ enjoys the (asynchronous) fellow traveler property if there is a constant $k$ such that for each $u, w \in L$, $w$ and $u$ (asynchronously) $k$-fellow travel. 
\end{definition}

\begin{definition}
A group $G$ is \emph{asynchronously automatic} if there is a regular language $L$ surjecting to $G$ which is finite to one, i.e., no element of $G$ has an infinite preimage, and which satisfies an asynchronous $k$-fellow traveler property for some $k \in \R$. 
\end{definition}

\begin{remark}
Asynchronously automatic groups can be defined in the same manner as automatic groups, in terms of an asynchronous finite automata, but this definition is better suited to the purposes of this paper. 
\end{remark}

So to prove that our group is asynchronously automatic we need to find a regular language that surjects onto $G$, this will be our normal form. Then we need to prove that the language has the asynchronous $k$-fellow traveller property. Before we get to the proof we just need one more quick definition that we will use throughout.  

\begin{definition}
We say a group is \emph{(asynchronously) combable} if it has a language surjecting to it which has the (asynchronous) fellow traveler property. 
\end{definition}

\section{Proof of Main Result}

In example \ref{example1}, instead of viewing our group as a single HNN extensions, we can view it as a double HNN extension of $\Z^2$, which is more helpful for our purposes. We can write $G = \langle a, t \mid at = ta \rangle$. Then we can first do a HNN extension by $b$, where $b^{-1}tb = at$:\[
G^* = \langle a, t, b \mid at = ta, b^{-1}tb = at \rangle.
\] Then we can perform another HNN extension by $c$ with the relation $c^{-1}tc = a^2t$: \[
G^{**} = \langle at = ta, b^{-1}tb = at, c^{-1}tc = a^2t \rangle.
\] In the first HNN extension the associated subgroups are $\langle t \rangle$ and $\langle at \rangle$. In the second extension our associated subgroups are $\langle t \rangle$ and $\langle a^2t \rangle$. 

\par Geometrically, our group can be viewed as a copy of $\Z^2$, which is our $at$-plane, and along the line $a=t$ we have a $b$-strip coming up, and along the $2a=t$ line there is a $c$ strip coming up. Then this repeats over and over. It has a tree-like structure. Once we go up or down on a $b, b^{-1}, c,$ or $c^{-1}$, we must go back down on the same strip to get back to the original plane. This can be seen in figure \ref{cayley}. 
\begin{figure}[h!]
    \centering
    \includegraphics[width=15cm]{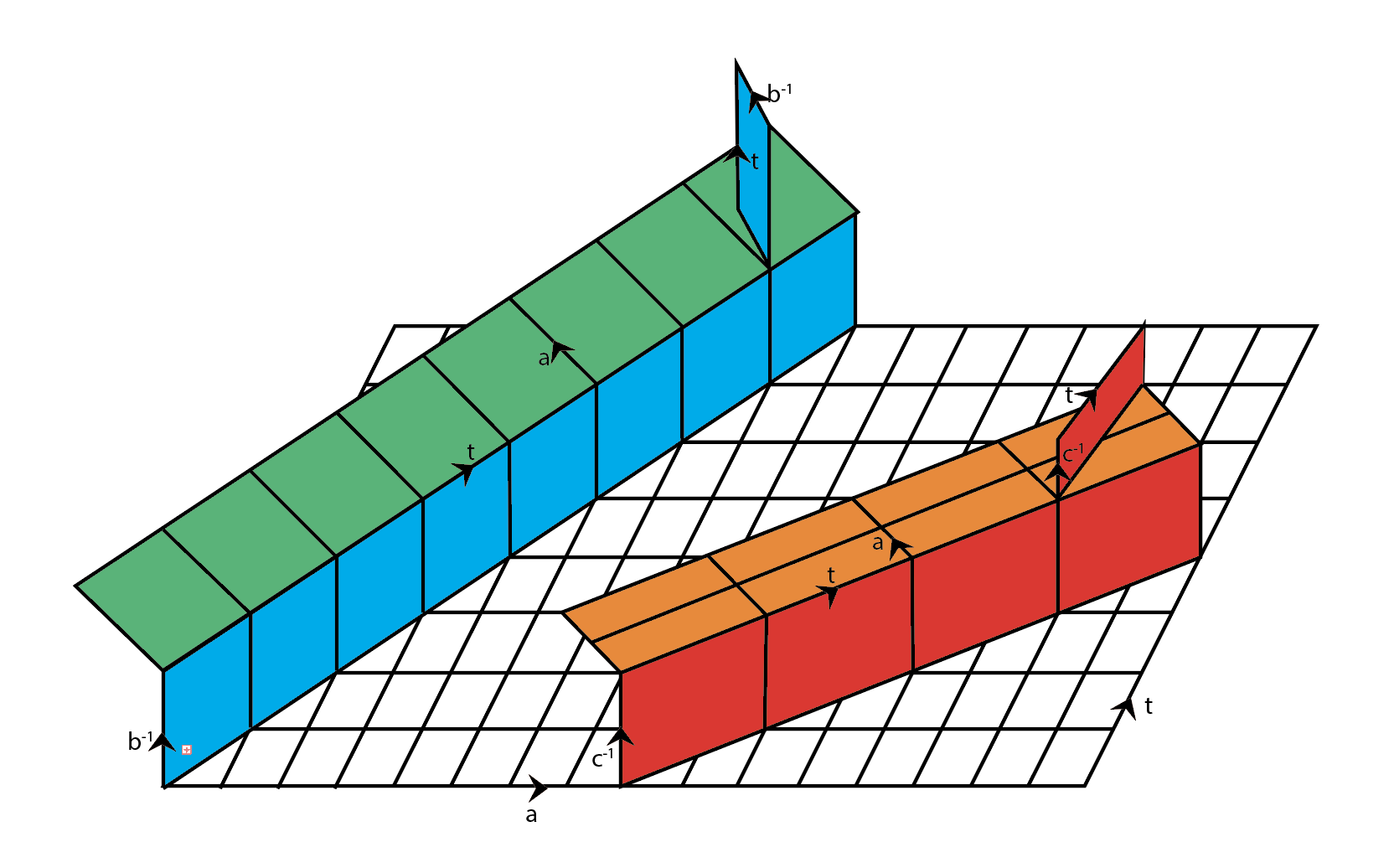}
    \caption{One layer of the Cayley graph of $G = \left<a, b, c, t \mid at=ta b^{-1}tb = at, c^{-1}tc = a^2t \right>$}
    \label{cayley}
\end{figure}

\begin{theorem}
The group \[
G = \left< a, t, b_1, b_2, \ldots b_k \mid at = ta, b_1^{-1}tb_1 = a^{n_1}t, b_2^{-1}tb_2 = a^{n_2}t, \ldots b_k^{-1}tb_k = a^{n_k}t \right>
\] 
for $n_1, n_2, \ldots n_k \in \Z$ is asynchronously automatic. 
\end{theorem}

\begin{proof}
For the sake of simplicity, we prove the theorem for the group $G = \left< a, t, b, c \mid at = ta, b^{-1}tb = a^mt, c^{-1}tc = a^nt \right>$.
 The proof is similar for an arbitrary number of HNN extensions.
 
First, we choose the set  $\set{a^i \mid i \in \Z}$ to be a set of coset representatives for each of our associated subgroups: $\left< t \right>$, $\left< a^m t \right>$, and $\left< a^nt \right>$. 
Clearly, the set $\set{a^i \mid i \in \Z}$ is a regular language.
Then we choose a regular language that will be a normal form for our copy of $\Z^2$ which will be $L_{\Z^2} =\set{a^it^j \mid i, j \in \Z}$.
The normal form $L_{\Z^2}$ is a synchronous 2-combing of $\Z^2$.
Now, by the normal form theorem we have that
\[
L = L_{\Z^2}
    \left(
        \{r a^i | r \in \{b,b^{-1},c,c^{-1}\}, i \in \Z \} 
    \right)^* 
\]
is a normal form.
So, a word $w \in L$ has the form 
$w = w_0 r_1 w_1 \cdots r_l w_l$ where $w_0 = a^{\alpha} t^{\beta}$, each $r_i \in \{b,b^{-1},c,c^{-1}\}$, and each $w_i = a^{\gamma}$.

We now show that $L$ has the asynchronous fellow traveler property.
Let $w$ be a word in $L$. 
We say two words $w, u$ have \emph{parallel stable letter structure} if $u = u_0 r_1 u_1 \cdots r_l u_l$ for the same set of stable letters, and each stable letter is reduced.
For our first case of the fellow-traveler property, let $r$ be a stable letter.
Then $wr \in L$ unless $w = w'r^{-1}$, in this case $wr = w' \in L$. 
Next, let $x \in \set{a^{\pm 1}, t^{\pm 1}, \left(a^mt\right)^{\pm 1}, \left(a^nt\right)^{\pm 1}}$.
We include everything besides $a^{\pm 1}, t^{\pm 1}$ for the purposes of induction.
Assume for all words $w'$ with $\abs{w'} < \abs{w}$ there is an $L$-word $u =_{G} w'x$ such that $w'$ and $u$ 2-fellow travel and $u, w'$ have parallel stable letter structure. 
\par We now go through the cases of adding the letter $x$ to the end of the word $w$.
\begin{enumerate}
    \item If $w = w_0 \in L_{\Z^2}$, then the $L$-word $wx$ will 2-fellow travel $w$ since $L_{\Z^2}$ is a 2-combing. 
    \item If $r_n = b^{-1}$, then $w = w'b^{-1}a^i$.
    \begin{itemize}
        \item If $x = a^{\pm 1}$, then $wa^{\pm 1} = w'b^{-1}a^{i \pm 1} \in L$.
        \item If $x = t^{\pm 1}$, then $wt^{\pm 1} = w'b^{-1}a^it^{\pm 1} = w'\left(a^mt\right)^{\pm 1}b^{-1}a^i = ub^{-1}a^i$ by induction. It can be seen in figure \ref{fig:k-fellow}, for the group $G = \left< a, b, c, t \mid at =ta b^{-1}tb = at, c^{-1}tc = a^2t \right>$ that $w'b^{-1}a^it$ and $w'a^2tb^{-1}a^i$ will 3-fellow travel. In general these words with $m+1$, travel. Indeed, starting at $w'$ first follow the path for $w'\left(a^mt\right)^{\pm 1}b^{-1}a^i$ along the $a^mt$, at this point the words $w'$ and $w'a^mt$ are at a distance of $m+1$ apart. Then simultaneously, go up $b^{-1}$ on both words, now they are a distance 1 apart, following the last part of both words, they never get more than a distance of 1 apart. Thus they asynchronously $m+1$-fellow travel. In general, the group $G = \left< a, t, b_1, b_2, \ldots b_k \mid at = ta, b_1^{-1}tb_1 = a^{n_1}t, b_2^{-1}tb_2 = a^{n_2}t, \ldots b_k^{-1}tb_k = a^{n_k}t \right>$ will asynchronously $\max\set{\left(n+1\right), \left(m+1\right)}$-fellow travel. A similar argument works in all other cases. 
        \begin{figure}
            \centering
            \includegraphics[width=12cm]{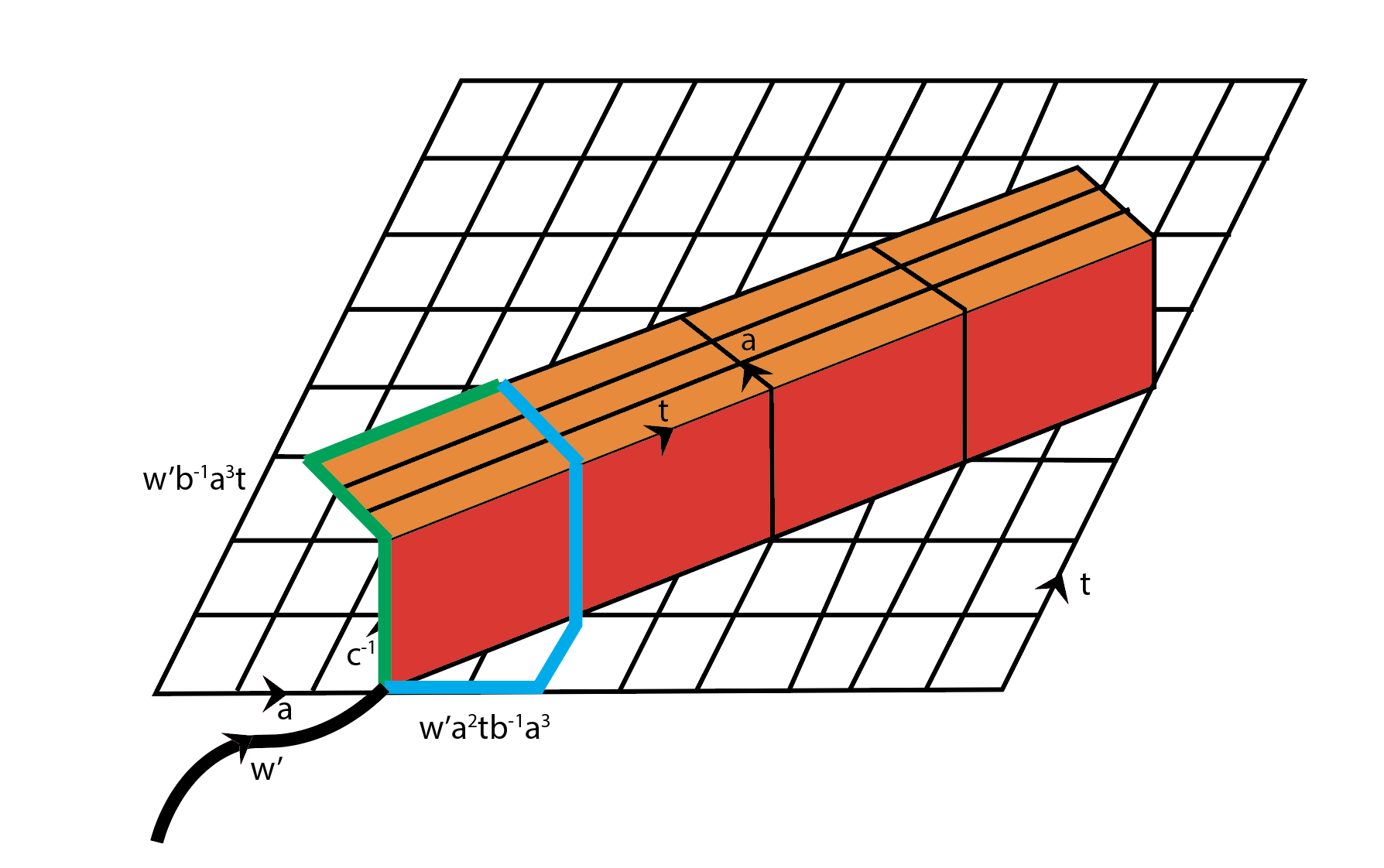}
            \caption{The words $w'b^{-1}a^3t$ and $w'a^2tb^{-1}a^3$ asynchronously 3-fellow travel. }
            \label{fig:k-fellow}
        \end{figure}
        \item If $x = \left(a^mt\right)^{\pm 1}$, $w\left(a^mt\right)^{\pm 1} = w'b^{-1}t^{\pm 1}a^{i \pm m} = w'\left(a^mt\right)^{\pm 1}b^{-1}a^{i \pm m} = ub^{-1}a^{i \pm m}$ by induction. 
        \item If $x = \left(a^nt\right)^{\pm 1}$, then this will be the same as above, with $m$ replaced by $n$. 
    \end{itemize}
    
    \item If $r_n = c^{-1}$, then this is the same as the case when $r_n = b^{-1}$, just replacing $m$ with $n$.

    \item If $r_n = b$, then $w = w'ba^i$.
    \begin{itemize}
        \item If $x = a^{\pm 1}$, then this will be in $L$.
        \item If $x = t^{\pm 1}$, then $wt^{\pm1} = w'bt^{\pm 1}a^i = w't^{\pm 1}ba^{i \mp m} = u ba^{i \mp m}$, by induction.
        \item If $x = \left(a^mt\right)^{\pm1}$, then $w\left(a^mt\right)^{\pm 1} = w'bt^{\pm 1}a^{i} = w't^{\pm1}ba^{i }$, by induction. 
        \item If $x = \left(a^nt\right)^{\pm 1}$, then this is the same as above, with $m$ replaced by $n$. 
    \end{itemize}
    \item If $r_n = c$, then this is the same as above with $m$ replaced by $n$. 
\end{enumerate}
\end{proof}

From this and chapter 7 of \cite{WP}, we get the following corollary.

\begin{corollary}
The word problem for \[G = \left< a, t, b_1, b_2, \ldots b_k \mid at = ta, b_1^{-1}tb_1 = a^{n_1}t, \ldots b_k^{-1}tb_k = a^{n_k}t \right> \text{ for } n_1, n_2, \ldots n_k \in \Z\] is solvable.
\end{corollary}

Lastly, we show that $G$ cannot be synchronously automatic with respect to the language we have constructed above. The following is theorem 3.3.4 from \cite{WP}.

\begin{theorem}\label{quasigeodesic-thm-wp}
Let $G$ be a group with automatic structure $(A, L)$. If there is a bound for the number of accepted representatives in $L$ of an element $g \in G$ as $g$ varies, then there exists $N$ such that any path in the Cayley graph of $G$ corresponding to an accepted string is an $(N, N)$-quasigeodesic. 
\end{theorem}

\begin{theorem}\label{L-not-quasi}
The language \[
L = L_{\Z^2}
    \left(
        \{r a^i | r \in \{b,b^{-1},c,c^{-1}\}, i \in \Z \} 
    \right)^* 
\] is not quasi-geodesic. 
\end{theorem}

\begin{proof}
Consider the geodesic element $c^{-k}t^j$. It is of length $k + j$, when we put it in its normal form it becomes \begin{align*}
    c^{-k}t^j =&\; c^{-k+1}a^nt^jc^{-1}t^{j-1}\\
    =&\; a^nt(c^{-1}a^n)^jc^{-1}t^{j-1}\\
    =&\; (a^nt)^j((c^{-1}a^n)^k)^jc^{-1}
\end{align*} which is of length $nj+j+kj+nkj+1$. 
\end{proof}

Thus we have: 
\begin{theorem}
The language \[
L = L_{\Z^2}
    \left(
        \{r a^i | r \in \{b,b^{-1},c,c^{-1}\}, i \in \Z \} 
    \right)^* 
\] is not synchronously automatic.
\end{theorem}
\begin{proof}
The language $L$ was constructed using the normal form theorem, and thus is a finite-to-one map. By theorem \ref{L-not-quasi}, our language is not quasi-geodesic, and thus by theorem \ref{quasigeodesic-thm-wp}, $L$ cannot be synchronously automatic. 
\end{proof}

\bibliographystyle{amsplain}
\bibliography{references}

\end{document}